\documentclass[11pt,twoside]{article}
\usepackage{latexsym}
\usepackage{amssymb,amsbsy,amsmath,amsfonts,amssymb,amscd}
\usepackage{amsthm}
\usepackage{epsfig, graphicx,subfigure}
\usepackage{enumitem}
\usepackage[active]{srcltx}

\usepackage{colortbl}
\newcommand{\commentout}[1]{}
\newcommand{\R}{\mathbb{R}}

\newcommand {\e}  {\varepsilon}

\newcommand {\Chi} {{\bf \raise 2pt \hbox{$\chi$}} }

\newcommand {\sgn} { {\rm sgn} }

\newcommand {\F} { {\mathcal F} }

\newcommand {\U} { {\mathcal U} }

\newcommand {\f}   {\frac}
\newcommand {\p}   {\partial}

\newcommand{\dis}{\displaystyle}
\newcommand{\PrPc}{PrP$^{\text{c}}$ }
\newcommand{\PrPSc}{PrP$^{\text{Sc}}$ }
\newcommand{\beq}{\begin{equation}}
\newcommand{\beqa} {\begin{array}{rl}}
\newcommand{\eeq}{\end{equation}}
\newcommand{\eeqa}{\end{array}}
\newtheorem{theorem}{Theorem}
\newtheorem{lemma}[theorem]{Lemma}
\newtheorem{definition}[theorem]{Definition}
\newtheorem{remark}[theorem]{Remark}
\newtheorem{proposition}[theorem]{Proposition}

\title{\LARGE Global stability for the prion equation with general incidence}

\author{P. Gabriel \thanks{Universit\'e de Versailles Saint-Quentin-en-Yvelines, Laboratoire de Math\'ematiques de Versailles, CNRS UMR 8100, 45 Avenue des \'Etats-Unis, 78035 Versailles cedex, France. Email: pierre.gabriel@uvsq.fr}}

\date{\today}

\begin{document}
\maketitle
\pagestyle{plain}
\pagenumbering{arabic}

\begin{abstract}
We consider the so-called prion equation with the general incidence term introduced in~\cite{Greer2}, and we investigate the stability of the steady states.
The method is based on the reduction technique introduced in~\cite{PG2}.
The argument combines a recent spectral gap result for the growth-fragmentation equation in weighted $L^1$ spaces and the analysis of a nonlinear system of three ordinary differential equations.
\end{abstract}


\

\noindent{\bf AMS Class. No.} 35B35, 35B40, 35Q92, 45K05, 92D25

\

\section{Introduction}

Prion diseases, also referred to as transmissible spongiform encephalopathies, are infectious and fatal neurodegenerative diseases. They include bovine spongiform encephalopathy in cattle, scrapie in sheep, and Creutzfeld-Jakob disease in human. 
It is now widely admitted that the agent responsible for these diseases, known as prion, is a protein which has the ability to self-replicate by an autocatalytic process~\cite{Griffith, Prusiner}.
The infectious prion, called \PrPSc for Prion Protein Scrapie, is a misfolded form of a normally shaped cellular prion protein, the \mbox{\PrPc.}
The so-called {\it nucleated polymerization} was proposed by~\cite{Lansbury} as a conversion mechanism of \PrPc into \mbox{\PrPSc.}
According to this theory the \PrPSc is in a polymeric form and the polymers can lenghten by attaching \PrPc monomers and transconforming them into \PrPSc.
To understand more qualitatively this mechanism, a mathematical model consisting in a infinite number of coupled ordinary differential equations (ODEs) was introduced in~\cite{Masel}.
Then a partial differential equation (PDE) version of this model was proposed in~\cite{Greer} (see also~\cite{DGL} for a rigourous derivation).
This equation, known as the prion equation, was studied in various works in the last few years~\cite{Engler,Pruss,SimonettWalker,Walker,LW,CL1,CL2,PG}.
A more general model including general incidence of the total population of polymers on the polymerization process and a coagulation term was proposed in~\cite{Greer2}.

In the present work we propose to investigate the prion equation with general incidence, but without coagulation, which writes
\begin{equation}
\label{eq:prion_gen_inc}
\left\{
\begin{array}{rll}
\f{d}{dt}V(t)&=&\displaystyle \lambda -\delta V(t) - \f{V(t)}{1+\omega\int x^pu}\int_{0}^{\infty} \tau(x) u(t,x) \; dx\,,
\vspace{.2cm}\\
\p_t u(t,x) &=& \displaystyle - \f{V(t)}{1+\omega\int x^pu}\,\p_x\big(\tau(x) u(t,x)\big) - \mu(x) u(t,x) + {\mathcal F} u(t,x),
\end{array} \right.
\end{equation}
where $\F$ defined by
\[\mathcal F u(x):=2\int_x^\infty \beta(y)\kappa(x,y)u(y)\,dy-\beta(x)u(x)\]
is the fragmentation operator.
Dynamics~\eqref{eq:prion_gen_inc} is subjected to nonnegative initial conditions $V_0$ and $u_0(x).$
The unknown $V(t)$ represents the quantity of \PrPc monomers at time $t$ while $u(t,x)$ is the quantity of \PrPSc polymers of size $x.$
The \PrPc is produced by the cells with the rate $\lambda$ and degraded with the rate $\delta.$
The \PrPSc polymers have a death rate $\mu(x)$ and they can break into two smaller pieces with the fragmentation rate $\beta(x).$
The kernel $\kappa(x,y)$ gives the size distribution of the fragments.
The ``general incidence'' corresponds to the term $\f{1}{1+\omega\int x^pu}$ in front of the polymerization rate $\tau(x),$ with $\omega\geq0$ and $p\geq0.$
The case $\omega=0$ corresponds to the mass action law, {\it i.e.} the original model without general incidence.
The more interesting case $\omega>0$ corresponds to the case when the total population of polymers induces a saturation effect on the polymerization process.
In~\cite{Greer2} the parameter $p$ is equal to 1, meaning that the saturation is a function of the total number of polymerized proteins.
To be more general and to take into account the fact that the polymers are not necessarily linear fibrils but can have more complex spatial structure (see~\cite{Masel}), we consider in our study any parameter $p\geq0.$
In~\cite{Greer2}, the polymerization rate $\tau(x)$ is supposed to be independant of $x.$
But some works~\cite{PG,Silveira} indicate that the polymerization ability, which relies on the infectivity of a polymer, may depend on its size.
For mathematical convinience in our work we assume that this dependence is linear
\begin{equation}\label{as:tau}
\tau(x)=\tau x\qquad(\tau>0).
\end{equation}
Notice that for such a function $\tau(x)$ there is no need of a boundary condition at $x=0$ for the equation on $u(t,x).$
In~\cite{Greer2} they restrict their study to linear global fragmentation rates $\beta(x)$ and to the homogeneous fragmentation kernel $\kappa(x,y)=1/y.$
Together with the assumption of a constant death term $\mu,$ it allows them to reduce the PDE model to a system of three ODEs.
Here we keep the assumption of a constant death term
\begin{equation}\label{as:mu}
\mu(x)\equiv\mu>0,
\end{equation}
but we consider more general global fragmentation rates
\begin{equation}\label{as:beta}
\beta(x)=\beta x^\gamma\qquad(\beta,\gamma>0)
\end{equation}
and more general (self-similar) fragmentation kernel
\begin{equation}\label{as:kappa}
\kappa(x,y)=\f{1}{y}\wp\Bigl(\f xy\Bigr)
\end{equation}
where $\wp(z)$ is a smooth function defined on $[0,1].$
To ensure the conservation of the total number of \PrPSc monomers during the fragmentation, the operator $\F$ must verify $\int_0^\infty x\mathcal Fu(x)\,dx=0$ for any function $u.$
This property is satisfied under the following assumption on $\wp$
\begin{equation}\label{as:wp}
2\int_0^1 z\wp(z)\,dz=1.
\end{equation}
Condition~\eqref{as:wp} is fulfilled for $\wp$ a symmetric, in the sense that $\wp(z)=\wp(1-z),$ probability measure.
We additionnally suppose that the derivative of $\wp$ satisfies
\begin{equation}\label{as:wp'}\int_0^1|\wp'(z)|\,dz<+\infty.\end{equation}

Our study of Equation~\eqref{eq:prion_gen_inc} is performed in the space $\R\times X,$ where $X:=L^1(\R_+,dx)\cap L^1(\R_+,x^ r dx)$ with $r>1.$
More precisely we work in the positive cone $\R_+\times X_+$ which is invariant under the dynamics~\eqref{eq:prion_gen_inc}.
We take $r\geq p$ in order to have $L^1(\R_+,x^pdx)\subset X,$ so that the general incidence term is well defined.
The space $X$ is a Banach space for the natural norm $\|\cdot\|_X=\|\cdot\|_0+\|\cdot\|_r$ where $\|u\|_\alpha=\int_0^\infty |u(x)|x^\alpha dx.$
But for a part of our study, we also need to consider the weaker norm $\|\cdot\|_1$ on $X.$

Denote by $X_{\rm w}$ the space $X$ endowed with its weak topology.
The solutions of Equation~\eqref{eq:prion_gen_inc} are understood in the following weak sense.
\begin{definition}\label{def:sol}
Given $V_0>0$ and $u_0\in X_+,$ we call $(V,u)$ a (global) weak solution to Equation~\eqref{eq:prion_gen_inc} if
\begin{enumerate}[label=(\roman*),itemsep=0pt]
\item $V\in C^1(\R_+)$ is a non-negative solution to
\[\dot V=\lambda-\Bigl[\delta+\frac{\tau\int xu}{1+\omega\int x^pu}\Bigr]V,\]
\item $u\in C(\R_+,X_{\rm w})\cap L^1_{loc}(\R_+,L^1(x^\gamma\,dx))$ and for all $t>0,$ $u(t,\cdot)\in X_+,$
\item for all $t>0$ and $\varphi\in W^{1,\infty}(\R_+)$ there holds
\begin{align*}
\int_0^\infty u(t,x)&\varphi(x)\,dx=\int_0^\infty u_0(x)\varphi(x)\,dx+\tau\int_0^t V(s)\frac{\int_0^\infty xu(s,x)\varphi'(x)\,dx}{1+\omega\int_0^\infty u(s,x)\,dx}\,ds\\
& -\mu\int_0^t\int_0^\infty u(s,x)\varphi(x)\,dxds
+ \beta\int_0^t\int_0^\infty x^\gamma u(s,x)\Bigl[2\int_0^1\varphi(zx)\wp(z)\,dz-\varphi(x)\Bigr]\,dxds.
\end{align*}
\end{enumerate}
\end{definition}

The question of the existence and uniqueness of solutions is addressed in~\cite{SimonettWalker,Walker,LW,EscoMischler3} for very similar equations.
In the present paper we are interested in the long time behavior of the solutions to Equation~\eqref{eq:prion_gen_inc} -- in the sense of Definition~\ref{def:sol} -- , assuming their existence.

\

We easily check that $(\bar V=\frac\lambda\delta,0)$ is a steady state of our equation.
We call this trivial steady state the disease free equilibrium (DFE) since there is no polymerized proteins in this situation ($u\equiv0$).
A natural question is to know whether there exist endemic equilibria (EE), namely steady states $(V_\infty,u_\infty)\in\R_+\times X_+^*$ where $X_+^*=X_+\setminus\{0\}.$
For an EE, we get by testing the equation on $u_\infty$ against $x$ and using the relation $\int_0^\infty x\mathcal Fu_\infty(x)\,dx=0$ that
\begin{equation}\label{eq:EE1}
\frac{V_\infty\tau}{1+\omega\int x^pu_\infty}=\mu,
\end{equation}
and then $u_\infty$ is a positive solution to
\begin{equation}\label{eq:EE2}
\mu\bigl(xu_\infty(x)\bigr)'+\mu u_\infty(x)=\mathcal F u_\infty(x).
\end{equation}
The existence of an EE as well as the stability of the DFE depend on the basic reproduction rate $\mathcal R_0$ of Equation~\eqref{eq:prion_gen_inc}, which indicates the average number of new infections caused by a single infective introduced to an entirely susceptible population.
To find this parameter $\mathcal R_0,$ we linearize the equation on $u$ about the DFE $(\bar V,0)$ and we test the resulting equation against $x$ to obtain
\[\frac{d}{dt}\int_0^\infty x u(t,x)\,dx\simeq \bar V\tau\int_0^\infty x u(t,x)\,dx-\mu \int_0^\infty x u(t,x)\,dx.\]
We deduce that $\mathcal R_0$ is given by
\[\mathcal R_0=\frac{\bar V\tau}{\mu}=\f{\lambda\tau}{\delta\mu}.\]
It is worth noticing that this parameter does not depend on the fragmentation coefficients $\beta,$ $\gamma,$ and $\wp.$
We can now summarize the results of the paper in the following main theorem. 

\begin{theorem}\label{th:summary}
If $\mathcal R_0\leq1,$ the unique equilibrium in $\R_+\times X_+$ is the DFE.
It is globally asymptotically stable for the norm $|V|+\|u\|_1.$

If $\mathcal R_0>1,$ then there exists a unique EE which coexists with the DFE.
The EE is locally stable for the norm $|V|+\|u\|_X$, and the nontrivial trajectories cannot approach the DFE in the sense that
\[u_0\not\equiv0\quad\implies\quad\liminf_{t\to+\infty}\int_0^\infty xu(t,x)\,dx>0.\]
In the case when $p\geq1$ and $\delta\geq\mu,$ the EE is globally  asymptotically stable in $\R_+\times X_+^*$ for the norm $|V|+\|u\|_X.$
\end{theorem}

\

The paper is organized as follows:
In Section~\ref{se:reductionODE} we explain the method which allows to reduce Equation~\eqref{eq:prion_gen_inc} to a system of ODEs, and
in Section~\ref{se:proof} we take advantage of this reduction to prove Theorem~\ref{th:summary}.

\

\section{Reduction to a system of ODEs}\label{se:reductionODE}

As suggested by Equation~\eqref{eq:EE2}, we use the properties of the linear growth-fragmentation equation
\begin{equation}\label{eq:linearGF}
\p_tu(t,x)+\mu\,\p_x\bigl(xu(t,x)\bigr)+\mu u(t,x)=\F u(t,x).
\end{equation}
This equation is also known as the {\it self-similar fragmentation equation} (see~\cite{EscoMischler3,CCM10,CCM11,GS,MS}).
Using Assumption~\eqref{as:wp} we obtain (at least formally) that Equation~\eqref{eq:linearGF} preserves the mass
\begin{equation}\label{eq:mass_conserv}
\forall t\geq0,\qquad \int_0^\infty xu(t,x)\,dx=\varrho_0:=\int_0^\infty xu_0(x)\,dx.
\end{equation}

\

Under Assumptions~\eqref{as:tau}-\eqref{as:wp}, this equation admits a unique (up to normalization) positive steady state $\U(x)$
(see~\cite{EscoMischler3,DG}), {\it i.e.} a unique $\U\in L^1(\R_+,x\,dx)$ satisfying
\[\mu \bigl(x\,\U(x)\bigr)'+\mu\,\U(x)=\F\U(x),\qquad \U(x)>0,\qquad \int_0^\infty x\,\U(x)\,dx=1.\]
This steady state belongs to $L^1(\R_+,x^\alpha dx)$ for any $\alpha\geq0,$ so it belongs to $X_+.$
The convergence of the solutions to this equilibrium has been investigated in~\cite{EscoMischler3,MMP2} and recent results give the exponential relaxation under some assumptions and in suitable spaces (see~\cite{PR,LP,CCM10,CCM11,BCG,GS,MS}).
Here we use the spectral gap result recently proved in~\cite{MS} under the assumption that $\wp$ is a smooth function satisfying Assumption~\eqref{as:wp'}.

\begin{theorem}[\cite{MS}]\label{th:MS}
Under Assumptions~\eqref{as:tau}-\eqref{as:wp'}, there exist $a>0$ and $C>0$ such that
\begin{equation}\label{eq:expo_decay}
\forall u_0\in X,\ \forall t\geq0,\qquad\|u(t,\cdot)-\varrho_0\,\U\|_X\leq Ce^{-at}\|u_0-\varrho_0\,\U\|_X.
\end{equation}
\end{theorem}

\

The method we use to prove Theorem~\ref{th:summary} is based on a (time dependent) self-similar change of variable introduced in~\cite{PG2} which allows to transform a solution of the prion equation into a solution to the linear growth-fragmentation equation.
Then we combine the spectral gap result~\eqref{eq:expo_decay} with an asymptotic analysis of the change of variable to get the long time behavior of Equation~\eqref{eq:prion_gen_inc}.

\

The change of variable is defined as follows.
Starting from $u(t,x)\geq0$ and $V(t)\geq0$ solution to Equation~\eqref{eq:prion_gen_inc} we define for $k:=\gamma^{-1}$
\[v(h(t),x):=W^k(t)u(t,W^k(t)x)\,e^{\mu(t-h(t))}\]
with $W$ solution to
\[\dot W=\gamma W\biggl(\frac{\tau}{1+\omega\int x^pu}V-\mu W\biggr),\]
and $h$ the solution to $\dot h=W,\ h(0)=0.$
We choose $W(0)=1$ to have $v(t=0,\cdot)=u_0.$
Since $V$ is positive we have $\dot W\geq-\gamma\mu W^2$ and so $W\geq\f{1}{1+\gamma\mu t}.$
As a consequence $h(t)\geq\f1{\gamma\mu}\ln(1+\gamma\mu t)\to+\infty$ when $t\to+\infty$ so $h$ is a bijection of $\R_+$ and $v(t,\cdot)$ is well defined for all $t\geq0.$
We can check that $v$ is a solution to the linear equation~\eqref{eq:linearGF}.
Then the convergence result of Theorem~\ref{th:MS} ensures that
\[v(t,x)\xrightarrow[t\to+\infty]{\|\cdot\|_X}\varrho_0\,\U(x).\]
We deduce, for $\alpha\in[0,r],$ the equivalence
\[\int_0^\infty x^\alpha u(t,x)\,dx\underset{t\to+\infty}{\sim}\varrho_0M_\alpha W^{k\alpha}(t)\,e^{\mu(h(t)-t)}\]
where $M_\alpha=\int_0^\infty x^\alpha\U(x)\,dx.$
This equivalence allows us to obtain an (asymptotically) closed system of ODEs which provides the behavior of the change of variables.
Define $Q(t)=\varrho_0\,e^{\mu(h(t)-t)}$ which satisfies
\[\dot Q=\mu Q(W-1).\]
Then denoting $f(I)=\f{\tau}{1+\omega M_pI}$ we have
\[\dot W\underset{t\to+\infty}{\sim}\gamma W\bigl(f(W^{kp}Q)V-\mu W\bigr)\]
and, since $M_1=1$ by definition of $\U,$
\[\dot V\underset{t\to+\infty}{\sim}\lambda-V\bigl(\delta+f(W^{kp}Q)W^kQ\bigr).\]
To make these equivalences more precise, we define for $\alpha\geq0$
\[\e_\alpha(t):=\f{\int x^\alpha u(t,x)\,dx}{M_\alpha Q(t)W^{k\alpha}(t)}-1=\f1{M_\alpha}\int(\varrho_0^{-1}v(h(t),x)-\U(x))x^\alpha\,dx\]
The following Lemma ensures that $\e_\alpha(t)\to0$ when $t\to+\infty$ if $\alpha\in[0,r].$
\begin{lemma}\label{lm:eps_bound}
For any $\alpha\in[0,r],$ there exists $C>0$ such that
\[|\e_\alpha(t)|\leq C\,\|\varrho_0^{-1}u_0-\U\|_X\,e^{-ah(t)}.\]
\end{lemma}
\begin{proof}
Using Theorem~\ref{th:MS} we have
\begin{align*}
|\e_\alpha(t)|&\leq M_\alpha^{-1}\int|\varrho_0^{-1}v(h(t),x)-\U(x)|x^\alpha\,dx\nonumber\\
&\leq M_\alpha^{-1}\|\varrho_0^{-1}v(h(t),\cdot)-\U\|_X\nonumber\\
&\leq C\,\|\varrho_0^{-1}u_0-\U\|_X\,e^{-ah(t)}.
\end{align*}
\end{proof}
For $\alpha=1$ we even have, using $M_1=1$ and the mass conservation law~\eqref{eq:mass_conserv}, that
\[\forall t\geq0,\qquad\e_1(t)=\int(\varrho_0^{-1}v(h(t),x)-\U(x))x\,dx=0\]
and as a consequence $\int_0^\infty x u(t,x)\,dx=W^k(t)Q(t).$
Setting $f(\e;I)=f((1+\e)I),$ we get that $(V,W,Q)$ is solution to the sytem
\begin{equation}\label{eq:systODE1}
\left\{\begin{array}{rcl}
\dot V&=&\dis\lambda-V\Bigl(\delta+f\bigl(\e_p;W^{kp}Q\bigr)W^kQ\Bigr),
\vspace{.2cm}\\
\dot W&=&\dis\gamma W\Bigl(f\bigl(\e_p;W^{kp}Q\bigr)V-\mu W\Bigr),
\vspace{.2cm}\\
\dot Q&=&\dis \mu Q\bigl( W-1\bigr),
\end{array}\right.
\end{equation}
with the initial condition $(V_0,W_0,Q_0)=(V_0,1,\varrho_0).$
Defining the relevant quantity $P(t)=W^k(t)Q(t)$ and using it instead of $Q$ as an unknown we obtain the other system
\begin{equation}\label{eq:systODE2}
\left\{\begin{array}{rcl}
\dot V&=&\dis\lambda-V\Bigl(\delta+f\bigl(\e_p;W^{k(p-1)}P\bigr)P\Bigr),
\vspace{.2cm}\\
\dot W&=&\dis\gamma W\Bigl(f\bigl(\e_p;W^{k(p-1)}P\bigr)V-\mu W\Bigr),
\vspace{.2cm}\\
\dot P&=&\dis P\Bigl(f\bigl(\e_p;W^{k(p-1)}P\bigr)V-\mu\Bigr).
\end{array}\right.
\end{equation}

\begin{remark}[Interpretation of $V,\ W,\ Q$ and $P$]
By definition we have that $V(t)$ is the number of monomeric proteins (\PrPc).
The relation $P(t)=\int_0^\infty xu(t,x)\,dx$ means that $P(t)$ represents the number of polymerized proteins (\PrPSc).
The unknown $Q(t)$ represents roughly the total number of polymers
\[\int_0^\infty u(t,x)\,dx=(1+\e_0(t))M_0Q(t),\]
and $W(t)$ is related to the mean size of the polymers
\[W^k(t)=(1+\e_0(t))M_0\,\f{\int xu(t,x)\,dx}{\int u(t,x)\,dx}.\]
\end{remark}
Another relevant quantity is $Y=V+P,$ the total number of proteins (\PrPc + \PrPSc). We have a system of ODEs satisfied by $(Y,Q,P)$:
\begin{equation}\label{eq:systODE3}
\left\{\begin{array}{rcl}
\dot Y&=&\dis\lambda-\delta Y+(\delta-\mu)P,
\vspace{.2cm}\\
\dot Q&=&\dis \mu Q\Bigl( P^\gamma Q^{-\gamma}-1\Bigr),
\vspace{.2cm}\\
\dot P&=&\dis P\Bigl(f\bigl(\e_p;P^pQ^{1-p}\bigr)(Y-P)-\mu\Bigr).
\end{array}\right.
\end{equation}
We will use alternatively formulations~\eqref{eq:systODE1},~\eqref{eq:systODE2} and~\eqref{eq:systODE3} to prove our main theorem.
These systems are not autonomous because of the term $\e_p.$
But the property that this term vanishes when $t\to+\infty$ (see Lemma~\ref{lm:eps_bound}) is sufficient to get the asymptotic behavior of the change of variable, as we will see in the next section.

\

\section{Proof of Theorem~\ref{th:summary}}\label{se:proof}

We divide the proof of Theorem~\ref{th:summary} into several propositions.

\begin{proposition}
There exists an EE if and only if $\mathcal R_0>1.$
This EE is unique and is explicitely given by
\[V_\infty=\f{\mu+\lambda\omega M_p}{\tau+\delta\omega M_p}\quad\text{and}\quad u_\infty=Q_\infty\,\U,
\qquad\text{with}\quad Q_\infty=\f{\mathcal R_0-1}{\f{\tau}{\delta}+\omega M_p}.\]
\end{proposition}

\begin{remark}
It is worth noticing that $V_\infty$ given in the proposition belongs to the interval $\bigl(\frac{\mu}{\tau},\bar V\bigr),$ recalling that $\bar V>\frac\mu\tau$ when (and only when) $\mathcal R_0>1.$
\end{remark}

\begin{proof}
We recall that an EE is a positive nontrivial steady state.
We deduce from Equation~\eqref{eq:EE2} and the uniqueness of $\U$ that the function $u_\infty$ of an EE is positively colinear to $\U,$ {\it i.e.} $u_\infty=Q_\infty\,\U$ with $Q_\infty>0.$
Then using the equation on $V$ at the equilibrium and Equation~\eqref{eq:EE1} we get that $(V_\infty,Q_\infty)$ is solution to the system
\begin{equation}\left\{\begin{array}{l}\label{eq:ODE_EE}
\lambda=\delta V_\infty+\frac{\tau V_\infty Q_\infty}{1+\omega M_p Q_\infty},
\vspace{2mm}\\
\mu=\frac{\tau V_\infty}{1+\omega M_pQ_\infty}.
\end{array}\right.\end{equation}
We easily check that this system has a unique solution different from $(\bar V,0),$ given by
\[V_\infty=\f{\mu+\lambda\omega M_p}{\tau+\delta\omega M_p},\qquad Q_\infty=\f{\mathcal R_0-1}{\f{\tau}{\delta}+\omega M_p}.\]
The value of $Q_\infty$ is positive if and only if $\mathcal R_0>1.$
\end{proof}

\

Now we give a useful lemma about the boundedness of $V,\ P$ and $W.$

\begin{lemma}\label{lm:uniform_bounds}
Any solution to Equation~\eqref{eq:systODE1} with $(V_0,W_0,P_0)\in\R_+\times\R_+^*\times\R_+$ satisfies
\[\exists K_0>0,\quad \forall t\geq0,\quad V(t)+P(t)\leq K_0,\]
\[\exists K_2>K_1>0,\quad\forall t\geq0,\quad K_1\leq W^{k(p-1)}(t)\leq K_2.\]
\end{lemma}

\begin{proof}
We start from
\[\f{d}{dt}(V+P)=\lambda-\delta V-\mu P\leq\lambda-\min(\delta,\mu)(V+P)\]
which ensures by the Gr\"onwall lemma the global boundedness of $V+P.$
Then from
\[\f{d}{dt}W=\f Wk\left(f\bigl(\e_p;W^{k(p-1)}P\bigr)V-\mu W\right)\leq\f Wk\left(\tau K_0-\mu W\right)\]
we get the global boundedness (from above) of $W$ by $\overline W>0.$
From
\[\f{d}{dt}V\geq\lambda-V(\delta+\tau K_0)\]
we obtain that $\liminf_{t\to+\infty}V(t)\geq\f{\lambda}{\delta+\tau K_0}>0.$
Then if $p\geq1$ we deduce from
\[\f{d}{dt}W\geq\f Wk\left(f\bigl(\e_p;\overline W^{k(p-1)}K_0\bigr)V-\mu W\right)\]
that $\liminf_{t\to+\infty}W(t)\geq \mu^{-1}f\bigl(\overline W^{k(p-1)}K_0\bigr)\liminf_{t\to+\infty}V(t)>0$ since $\lim_{t\to+\infty}\e_p(t)\to0.$
For the case $p<1$ we write
\[\f{d}{dt}W\geq\f {W f\bigl(\e_p;W^{k(p-1)}P\bigr)}{k\tau}\left(\tau V-\mu W\bigl(1+\omega (1+\e_p) M_p W^{k(p-1)}K_0\bigr)\right)\]
and we define
\[g(W)=W\bigl(1+\omega M_p W^{k(p-1)}K_0\bigr).\]
The function $g$ is continuous and satisfies $g(0)=0$ and $\lim_{W\to+\infty}g(W)=+\infty,$ so there exists $W_1>0$ such that $g(W_1)=\f\tau\mu\liminf V$ and for all $W<W_1,$ $g(W_1)<\f\tau\mu\liminf V.$
Since $\e_p\to0$ when $t\to+\infty,$ we deduce that $\liminf_{t\to+\infty}W\geq W_1>0.$
Finally we have proved the existence of $K_1$ and $K_2$ because $W_0=1>0$ and $W$ cannot vanish in finite time.
\end{proof}

\begin{proposition}\label{prop:R0<=1}
If $\mathcal R_0\leq1$, then the DFE is globally asymptotically stable for the norm $|V|+\|u\|_1.$
\end{proposition}

\begin{proof}

Define $\tilde V=V-\bar V.$
The stability of the DFE in norm $|V|+\|u\|_1=|V|+|P|$ is ensured by the Lyapunov functional
\[\f{d}{dt}\biggl(\bar VP(t)+\f{\tilde V^2(t)}{2}\biggr)=-\Bigl(\mu-f\bar V\Bigr)\bar VP-\delta\tilde V^2-\tilde V^2 fP\leq0.\]
It remains to prove the global attractivity.

\noindent{\bf First case: $\mathcal R_0<1.$}

We have
\[\f{d}{dt}\biggl(\bar VP+\f{\tilde V^2}{2}\biggr)\leq-\Bigl(\mu-f\bar V\Bigr)\bar VP-\delta\tilde V^2.\]
Since $\mathcal R_0<1$ we have $\mu>\tau\bar V$ and
\[\f{d}{dt}\biggl(\bar VP+\f{\tilde V^2}{2}\biggr)\leq-\min(\mu-\tau\bar V,2\delta)\biggl(\bar VP+\f{\tilde V^2}{2}\biggr).\]
We deduce the exponential convergence from the Gr\"onwall lemma.

\

\noindent{\bf Second case: $\mathcal R_0=1.$}

When $\mathcal R_0=1$ we only have
\[\f{d}{dt}\biggl(\bar VP+\f{\tilde V^2}{2}\biggr)
\leq-\tau\mu\Bigl(1-\f{f}{\tau}\Bigr)\bar VP-\delta\tilde V^2-f\tilde V^2P\]
so we need to be more precise and estimate the value of $1-\frac f\tau.$
Using System~\eqref{eq:systODE2} we have
\[1-\f{f}{\tau}=\f{(1+\e_p)\omega M_p W^{k(p-1)}P}{1+(1+\e_p)\omega M_p W^{k(p-1)}P}.\]
From Lemma~\ref{lm:eps_bound} we can ensure the existence of a time $t_0\geq0$ such that $|\e_p(t)|\leq\frac12$ for all $t\geq t_0.$
Then using Lemma~\ref{lm:uniform_bounds} we get that for all $t> t_0$
\begin{align*}\f{d}{dt}\biggl(\bar VP+\f{\tilde V^2}{2}\biggr)&
\leq-\f{\tau\mu\,\omega M_p K_1}{2+3\omega M_p K_2K_0}\bar VP^2-\f{\delta}{K_0^2}\tilde V^4-\tau\tilde V^2P\\
&\leq-\min\left\{\f{\tau\mu\,\omega M_p K_1}{2+3\omega M_pK_2K_0}\bar V^{-1},\f{4\delta}{K_0^2},\tau\bar V^{-1}\right\}\Bigl(\bar VP+\f{\tilde V^2}{2}\Bigr)^2.
\end{align*}
After integration this gives for $t\geq t_0$
\[\bar VP+\f{\tilde V^2}{2}\leq\f{1}{\f{1}{\bar VP(t_0)+\tilde V^2(t_0)/2}+C t}\xrightarrow[t\to+\infty]{}0\]
where the constant $C=\min\left\{\f{\tau\mu\,\omega M_p K_1}{2+3\omega M_pK_2K_0}\bar V^{-1},\f{4\delta}{K_0^2},\tau\bar V^{-1}\right\}>0.$
 \end{proof}

\

\begin{proposition}\label{prop:R0>1}
If $\mathcal R_0>1,$ then the unique EE is locally asymptotically stable for the norm $|V|+\|u\|_X.$
\end{proposition}

\begin{proof}
We want to prove
\begin{equation}\label{stability}
\forall\epsilon>0,\ \exists\eta>0,\quad |V_0-V_\infty|+\|u_0-u_\infty\|_X<\eta\ \implies\ \forall t\geq0,\ |V(t)-V_\infty|+\|u(t,\cdot)-u_\infty\|_X<\epsilon.
\end{equation}

\medskip

\noindent{\bf Step \#1.} We start from the homogeneous form of Equation~\eqref{eq:systODE1} (obtained by replacing $\e_p$ by $0$) which writes
\begin{equation}\label{eq:ODE1_hom}\left\{\begin{array}{rcl}
\dot V&=&\dis\lambda-V\Bigl(\delta+f\bigl(W^{kp}Q\bigr)W^kQ\Bigr),
\vspace{.2cm}\\
\dot W&=&\dis\gamma W\Bigl(f\bigl(W^{kp}Q\bigr)V-\mu W\Bigr),
\vspace{.2cm}\\
\dot Q&=&\dis \mu Q\bigl( W-1\bigr).
\end{array}\right.\end{equation}
We easily check from Equation~\eqref{eq:ODE_EE} that $(V_\infty,1,Q_\infty)$ is the unique equilibrium of System~\eqref{eq:ODE1_hom}.
First we prove the linear stability of this equilibrium by using the Routh-Hurwitz criterion.
The Jacobian of System~\eqref{eq:ODE1_hom} about $(V_\infty,1,Q_\infty)$ is
\[Jac_{eq}=\left(\begin{array}{ccc}
-\delta-Qf(Q) & -kVQ\bigl(pQ f'(Q)+f(Q) \bigr) &  -V\bigl(Qf'(Q)+ f(Q)\bigr)\\
\gamma f(Q) & pQf'(Q)V-\gamma\mu &\gamma f'(Q)V\\
0 & \mu Q & 0
\end{array}\right).\]
where we have skipped the indices $_\infty$ for the sake of clarity.
To use the Routh-Hurwitz criterion, we compute the trace
\[T=-\delta - \gamma\mu -Qf(Q) +pVQf'(Q),\]
the determinant
\[D=\gamma\mu VQ\bigl(\delta f'(Q)-f^2(Q)\bigr),\]
and the sum of the three $2\times2$ principal minors
\[M=\gamma\mu\bigl(\delta+Qf(Q)-VQf'(Q)\bigr)-p\delta VQf'(Q)+VQf^2(Q).\]
We have $T<0,$ $D<0,$ and from
\[T<-\delta-\gamma\mu\qquad \text{and}\qquad M>-\gamma\mu VQf'(Q)+VQf^2(Q)\]
we obtain that $MT<D.$
By the Routh-Hurwitz criterion, we deduce that the steady state $(V_\infty,1,Q_\infty)$ is locally stable for Equation~\eqref{eq:ODE1_hom}.

\medskip

\noindent{\bf Step \#2.} By continuity of the function $f,$ the steady-state $(V_\infty,1,Q_\infty)$ is also stable for System~\eqref{eq:systODE1} provided that $\|\e_p\|_\infty$ is small enough, {\it i.e.}
\begin{equation}\label{stability1}\begin{array}{rl}
\forall\epsilon,\ \exists\eta>0,&\quad |V_0-V_\infty|+|W_0-1|+|Q_0-Q_\infty|+\|\e_p\|_\infty<\eta
\vspace{2mm}\\
&\hspace{20mm}\implies\ \forall t\geq0,\ |V(t)-V_\infty|+|W(t)-1|+|Q(t)-Q_\infty|<\epsilon.
\end{array}\end{equation}
We would like to replace $|W_0-1|+|Q_0-Q_\infty|+\|\e_p\|_\infty$ in~\eqref{stability1} by $\|u_0-u_\infty\|_X.$
From our choice of $W_0,$ we have $|W_0-1|=0.$
For $|Q_0-Q_\infty|$ we have $Q_0=\varrho_0$ and
\begin{equation}\label{eq:Q0_Qinf}
|\varrho_0-Q_\infty|=\left|\int_0^\infty(u_0(x)-u_\infty(x))x\,dx\right|\leq\|u_0-u_\infty\|_X.
\end{equation}
For the last term $\|\e_p\|_\infty$ we know from Lemma~\ref{lm:eps_bound} that
\[\|\e_p\|_\infty\leq M_p^{-1}\|\varrho_0^{-1}u_0-\U\|_X.\]
But using~\eqref{eq:Q0_Qinf} we also have
\begin{align*}
\|\varrho_0^{-1}u_0-\U\|_X&=\frac1{Q_\infty}\left\|\frac{Q_\infty}{\varrho_0}u_0-u_\infty\right\|_X\\
&\leq\frac1{Q_\infty}\Bigl(|Q_\infty-\varrho_0|+\|u_0-u_\infty\|_X\Bigr)\\
&\leq\f2{Q_\infty}\|u_0-u_\infty\|_X.
\end{align*}
At this stage we have proved that for all $\epsilon>0$ there exists $\eta>0$ such that
\begin{equation}\label{stability2}
|V_0-V_\infty|+\|u_0-u_\infty\|_X<\eta\ \implies\ \forall t\geq0,\ |V(t)-V_\infty|+|W(t)-1|+|Q(t)-Q_\infty|<\epsilon.
\end{equation}

\medskip

\noindent{\bf Step \#3.} It remains to deduce \eqref{stability} from~\eqref{stability2}.
We write
\[\|u(t,\cdot)-u_\infty\|_X\leq\|u(t,\cdot)-Q(t)W^{-k}(t)\U(W^{-k}(t)\cdot)\|_X+\|Q(t)W^{-k}(t)\U(W^{-k}(t)\cdot)-Q_\infty\U\|_X.\]
For the first term we have
\begin{align}
\|u(t,\cdot)-Q(t)W^{-k}(t)\U(W^{-k}(t)\cdot)\|_X&=Q(t)\int_0^\infty|\varrho_0^{-1}v(h(t),x)-\U(x)|(1+W^{kr}(t)x^r)\,dx\nonumber\\
&\leq \bigl(Q_\infty+|Q-Q_\infty|\bigr)\bigl(1+|W-1|\bigr)^{kr}\|\varrho_0^{-1}u_0-\U\|_X\nonumber\\
& \leq 2\bigl(1+\bigl|\frac{Q}{Q_\infty}-1\bigr|\bigr)\bigl(1+|W-1|\bigr)^{kr}\|u_0-u_\infty\|_X.\label{stability3}
\end{align}
For the second term we have by dominated convergence that
\begin{equation}\label{stability4}
\forall\epsilon>0,\ \exists\eta>0,\ \quad|W-1|+|Q-Q_\infty|<\eta\ \implies\ \|QW^{-k}\U(W^{-k}\cdot)-Q_\infty\U\|_X<\epsilon.
\end{equation}
Combining~\eqref{stability2}, \eqref{stability3} and~\eqref{stability4}, we obtain~\eqref{stability} and the proposition is proved.

\end{proof}

\

\begin{proposition}\label{prop:persistence}
If $\mathcal R_0>1$, then the trajectories cannot approach the DFE in the sense that
\[\liminf_{t\to+\infty}\int_0^\infty xu(t,x)\,dx>0.\]
\end{proposition}

\begin{proof}
We are in the case $\mathcal R_0>1$ so $\theta:=\tau\bar V-\mu>0.$

\noindent{\bf First case: $\forall t,\ V(t)\geq\bar V.$}
Using System~\eqref{eq:systODE2}, Lemma~\ref{lm:uniform_bounds} and Lemma~\eqref{lm:eps_bound} we have
\begin{align*}
\f{d}{dt}P&=P\bigl(f(\e_p;W^{k(p-1)}P)V-\mu\bigr)\\
&\geq P\bigl((f(\e_p;W^{k(p-1)}P)-\tau)V+\tau\bar V-\mu\bigr)\\
&=P\biggl(-\f{\tau\omega M_p(1+\e_p)W^{k(p-1)}P}{1+\omega M_p(1+\e_p)W^{k(p-1)}P}V+\theta\biggr)\\
&\geq P\biggl(-\tau\omega M_p(1+C\|\varrho_0^{-1}u_0-\U\|)K_1 K_0 P+\theta\biggr)
\end{align*}
and we deduce that
\[\liminf_{t\to+\infty} P\geq\f{\theta}{\tau\omega M_p(1+C\|\varrho_0^{-1}u_0-\U\|)K_1K_0}>0.\]
Remark that this case cannot hold since the positivity of the $\liminf P$ together with the equation on $V$ implies that $V$ becomes lower than $\bar V$ in finite time.
So we are always in the second case.

\

\noindent{\bf Second case: $\exists t_0\geq0,\ V(t_0)<\bar V.$}
Define the positive function $\tilde V(t)=\bar V-V(t)\ (\forall t\geq t_0,\ \tilde V(t)>0).$
As in~\cite{CL2} we compute, for $\alpha>0$ to be chosen later,
\[\f{d}{dt}\biggl(\f{P}{\tilde V^\alpha}\biggr)\geq\f{P}{\tilde V^\alpha}\bigl((\bar V-\tilde V)f-\mu\bigr)-\alpha\f{P}{\tilde V^\alpha}\Bigl(-\delta+\tau\bar V\f{P}{\tilde V}\Bigr).\]
We choose $\alpha$ large enough so that $\eta:=\alpha\delta-\mu>0.$
Denoting $R=P\tilde V^{-\alpha}$ we have
\[\dot R\geq R(\eta-\alpha\tau\bar V R^{1/\alpha}P^{1-1/\alpha})\]
and, choosing $\alpha\geq1,$
\begin{align*}\dot P&=P\f{f}{\tau}\biggl(\tau\bar V-\mu+\mu\Bigl(1-\f{\tau}{f}\Bigr)-\tau\Bigl(\f{P}{R}\Bigr)^{1/\alpha}\biggr)\\
&\geq P\f{f}{\tau}\biggl(\theta-\mu\omega M_pK_2P-\tau\Bigl(\f{P}{R}\Bigr)^{1/\alpha}\biggr)\\
&\geq P\f{f}{\tau}\biggl(\theta-\Bigl(\mu\omega M_pK_2K_0^{\f{\alpha-1}{\alpha}}+\f{\tau}{R^{\f1\alpha}}\Bigr)P^{1/\alpha}\biggr).
\end{align*}
The first inequality  tells us that
\[\underline R:=\liminf_{t\to+\infty} R\geq\biggl(\f{\eta}{\alpha\tau\bar V K_0^{1-1/\alpha}}\biggr)^\alpha>0.\]
Then the second inequality ensures that
\[\liminf_{t\to+\infty}P\geq\left(\f{\theta}{\mu\omega M_pK_2K_0^{\f{\alpha-1}{\alpha}}+\tau\underline R^{-\f1\alpha}}\right)^\alpha>0.\]
\end{proof}

\begin{proposition}\label{prop:p>=1}
In the case when $\mathcal R_0>1$ and additionnaly $p\geq1$ and $\delta\geq\mu,$ the EE is globally asymptotically stable for the norm $|V|+\|u\|_X.$
\end{proposition}

\begin{proof}
Consider the homogeneous form of System~\eqref{eq:systODE3} (by replacing $\e_p$ by $0$).
The matrix of partial derivatives has the sign pattern
\[\left[\begin{array}{ccc}
-&0&\sgn(\delta-\mu)\\
0&*&+\\
+&\sgn(p-1)&*
\end{array}\right].\]
In the case $p\geq1$ and $\delta\geq\mu,$ this indicates an irreducible cooperative system.
Then by Theorems 2.3.2, 4.1.1 and 4.1.2 on respective pages 18, 56 and 57 of~\cite{Smith}, the homogeneous form of System~\eqref{eq:systODE3} exhibits monotone dynamical flow and solutions must approach an equilibrium.
From Proposition~\ref{prop:persistence} the trajectories cannot approach the DFE when $\mathcal R_0>1,$ so they necessarily approach the EE.
Using the stability result of Proposition~\ref{prop:R0>1} we deduce the global asymptotic stability of the EE.

To conclude to the same result for the original System~\eqref{eq:systODE3}, we use the fact that $\e_p(t)\to0$ when $t\to+\infty$ and Lemma~4.2 in~\cite{PG2}.
\end{proof}

\

\section{Conclusion}

We have considered a prion model with less terms than in~\cite{Greer2}, but with more general coefficients.
Compared to the results in~\cite{Greer2} we have proved the global stability of the DFE in the critical case $R_0=1$ and the global asymptotic stability of the EE when the system is cooperative.

The results in Theorem~\ref{th:summary} remain valid for more general incidence functions $f$ provided that they are decreasing.
Indeed it has been proved in~\cite{PG2} that for increasing functions $f,$ periodic solutions can exist.
This indicates that Equation~\eqref{eq:prion_gen_inc} can exhibit various behaviors and their classification in the general case is still an open question.

\vspace{1cm}

\noindent{\bf Aknowledgment}

This work was supported by the french ANR project ``KIBORD'', ANR-13-BS01-0004-01.

\

%
%

\end{document}